\documentclass[11pt]{amsart}
\usepackage{amssymb,latexsym,graphicx, amscd}
\usepackage{float}
\newtheorem{theorem}{Theorem}[section]

\newtheorem{lemma}[theorem]{Lemma}

\newtheorem{definition}[theorem]{Definition}

\newtheorem{example}[theorem]{Example}

\newcommand\R{{\mathbb R}}

\newcommand{\CPb}{\overline{\mathbb{CP}}^{2}}
\newcommand{\CP}{{\mathbb{CP}}^{2}}

\def \CPb {\overline{\mathbb{CP}}^{2}}
\def \CP {{\mathbb{CP}}^{2}} 
\def \R {\mathbf{R}}

\def \Sig{\Sigma}

\def \- {\setminus}

\def\pt{\text{pt}}

\title{New Exotic $4$-Manifolds via Luttinger Surgery on Lefschetz Fibrations} 
              
\begin{document}

\author{Anar Akhmedov}
\address{School of Mathematics,
University of Minnesota,
Minneapolis, MN, 55455, USA}
\email{akhmedov@math.umn.edu}

\author{Kadriye Nur Saglam}
\address{School of Mathematics,
University of Minnesota,
Minneapolis, MN, 55455, USA}
\email{sagla004@math.umn.edu}

\begin{abstract} In \cite{A2}, the first author constructed the first known examples of exotic minimal symplectic $\CP\#5\CPb$ and minimal symplectic $4$-manifold that is homeomorphic but not diffeomorphic to $3\CP\#7\CPb$. The construction in \cite{A2} uses Y. Matsumoto's genus two Lefschetz fibrations on $M = \mathbb{T}^{2}\times \mathbb{S}^{2}\,\#4\CPb$ over $\mathbb{S}^2$ along with the fake symplectic $\mathbb{S}^{2} \times \mathbb{S}^{2}$ construction given in \cite{A1}. The main goal in this paper is to generalize the construction in \cite{A2} using the higher genus versions of Matsumoto's fibration constructed by Mustafa Korkmaz and Yusuf Gurtas on $M(k,n) = \Sigma_{k}\times \mathbb{S}^{2}\,\#4n\CPb$ for any $k \geq 2$ and $n = 1$, and $k \geq 1$ and $n \geq 2$, respectively. Using our building blocks, we also construct symplectic $4$-manifolds with the free group of rank $s \geq 1$ and various other finitely generated groups as the fundamental group. 

\end{abstract}

\maketitle

\section{Introduction}

The main goal of this paper is to exhibit a new family of simply connected minimal symplectic and infinitely many non-symplectic $4$-manifolds that is homeomorphic but not diffemorphic to $(2n+2k-3)\CP\#(6n+2k-3)\CPb$ for any $n \geq 1$ and $k \geq 1$. Our construction is a generalization of the first author's work in \cite{A2}, where he used Yukio Matsumoto's genus two Lefschetz fibrations on $M = \mathbb{T}^{2}\times \mathbb{S}^{2}\,\#4\CPb$ over $\mathbb{S}^2$, arising from the elliptic involution of the genus two surface with two fixed points, along with the fake symplectic $\mathbb{S}^{2} \times \mathbb{S}^{2}$ construction in \cite{A1} obtained via knot surgery along the fibered knots. In the aforementioned paper \cite{A2}, using the symplectic connected sum along the genus two surfaces, the first author has constructed the first known exotic minimal symplectic $\CP\#5\CPb$ and minimal symplectic $4$-manifold that is homeomorphic but not diffeomorphic to $3\CP\#7\CPb$. A straightforward generalization of these symplectic examples to an exotic $(2n-3)\CP\#(6n-3)\CPb$ (for $n \geq 4$), using Matsumoto's genus two fibration, can be found in \cite{ABBKP}, where the geography problem for simply connected minimal symplectic $4$-manifolds with signature less than equal $-2$ studied in details. Considering that Mustafa Kormaz's and Yusuf Gurtas' higher genus Lefschetz fibrations on $M(k,n) = \Sigma_{k}\times \mathbb{S}^{2}\,\#4n\CPb$, arising from the involution of the genus $2k + n - 1$ surface, are the generalizations of Matsumoto's genus two fibration, it was a natural question whether the similar construction can be carried out using these higher genus Lefschetz fibrations. The fundamental group of these Lefschetz fibrations are not abelian if $k \geq 2$ and the relations in the fundamental group coming from the vanishing cycles are far more complicated, thus the fundamental group computations become more difficult and delicate than in \cite{A2}. In this article, by carefully analysing the fundamental group, we overcome these difficulties and prove the following theorem:

\begin{theorem}\label{thm:main}
Let $M$\/ be \/ $(2n+2k-3)\CP\#(6n+2k-3)\CPb$ for any $n \geq 1$ and $k \geq 1$. There exists a new family of smooth closed simply-connected minimal symplectic\/ $4$-manifold and an infinite family of non-symplectic $4$-manifolds that are homeomorphic but not diffeomorphic to\/ $M$.    
\end{theorem}

One difference, though not essential, with the construction in \cite{A2} (see also \cite{A1}) is that we use the Luttinger surgery instead of the knot surgery. The Luttinger surgery will have some advantages for us in the computation of the fundamental groups, since the presentation of the fundamental groups look much simpler via Luttinger surgery. This is due to the fact that the relations introduced by our Luttinger surgeries all involves a single commutator relation expression, while the fundamental groups via knot surgery is more challenging to work with. Our construction can be also interpreted using the knot surgery. 

The second goal of our paper is to construct the symplectic $4$-manifolds with various fundamental groups (and with a small size) by applying the Luttinger surgeries to a certain family of Lefschetz fibrations. More general and stronger results along these lines are obtained in \cite{AO} and \cite{AZ}. 
 
\section{Mapping class group and Lefschetz fibrations}

\subsection{Mapping Class Groups} 

\par Let $\Sigma_{g}$ be an orientable $2$-dimensional closed and connected surface of genus $g>0$.

\begin{definition}

Let $Diff^{+}\left( \Sigma_{g}\right)$ denote the group of all orientation-preserving diffeomorphisms $\Sigma_{g}\rightarrow \Sigma_{g},$ and $ Diff_{0}^{+}\left(
\Sigma_{g}\right) $ be the subgroup of $Diff^{+}\left(\Sigma_{g}\right) $ consisting of all orientation-preserving diffeomorphisms $\Sigma_{g}\rightarrow \Sigma_{g}$ that are isotopic to the identity. \emph{The mapping class group} $M_{g}$ of $\Sigma_{g}$ is defined to be the group of isotopy classes of orientation-preserving diffeomorphisms of $\Sigma_{g}$, i.e.,
\[
M_{g}=Diff^{+}\left( \Sigma_{g}\right) /Diff_{0}^{+}\left(
\Sigma_{g}\right) .
\]
\end{definition}

\begin{definition}
\par Let $\alpha$ be a simple closed curve on $\Sigma_{g}$. A \emph{right handed Dehn twist} $t_\alpha$ about $\alpha$ is the isotopy class of a self-diffeomophism of $\Sigma_{g}$ obtained by cutting the surface $\Sigma_{g}$ along $\alpha$ and gluing the ends back after rotating one of the ends $2\pi$ to the right. 
\end{definition}

\begin{definition} \par Let $S(\Sigma_{g})$ denote be the set of all isotopy classes of simple closed curves in $\Sigma_{g}$. The \emph{intersection number} of two classes $[\alpha]$ and $[\beta]$ is $\mu([\alpha], [\beta]) = min \{ |a \cap b| | a \in [\alpha], b \in [\beta] \}$. 
\end{definition}

Let us mention a few basic facts about Dehn twists. First, notice that the conjugate of a Dehn twist is again a Dehn twist. Indeed, if $f: \Sigma_{g}\rightarrow \Sigma_{g}$ is an orientation-preserving diffeomorphism, then we can easily see that $f \circ t_\alpha \circ f^{-1} = t_{f(\alpha)}$. Second, if $\alpha$ and $\beta$ are two simple closed curves with $\mu([\alpha], [\beta]) = k \geq 0$, then $t_{\alpha }(\beta) = {\alpha}^{k}\beta$. In Section~\ref{sbb}, we will use some relations that hold in the mapping class group $M_{g}$. In order to provide a right framework for our discussion in Section~\ref{sbb}, we briefly review the theory of Lefschetz fibrations on $4$-manifolds below.

\subsection{Lefschetz Fibration}

\begin{definition}\label{defn:lef.fibration} Let $X$ be a compact, connected, oriented, smooth $4$-manifold. A \emph{Lefschetz fibration} on $X$ is a smooth map $f : X \longrightarrow \Sigma_{h}$,
where $\Sigma_{h}$ is a compact, oriented, smooth $2$-manifold of genus $h$, such that $f$ is surjective and each critical point of $f$ has an orientation preserving chart on which $f: \mathbb{C}^2 \longrightarrow \mathbb{C}$ is given by $f(z_1, z_2) = {z_1}^2 + {z_2}^2$.
\end{definition}

It is a consequence of the Sard's theorem that $f$ is a smooth fiber bundle away from finitely many critical points. Let us denote the critical points of $f$ by $p_1, \cdots, p_s$. The genus of the regular fiber of $f$ is defined to be the genus of the Lefschetz fibration. If a fiber of $f$ passes through critical point set $p_{1}, \cdots, p_{s}$, then it is called a singular fiber which is an immersed surface with a single transverse self-intersection. A singular fiber of the genus $g$ Lefschetz fibration can be described by its monodromy, i.e., an element of the mapping class group $M_{g}$. This element is a right-handed (or a positive) Dehn twist along a simple closed curve on $\Sigma_g$, called the \emph {vanishing cycle}. If this curve is a nonseparating curve, then the singular fiber is called \emph{nonseparating}, otherwise it is called \emph{separating}. For a genus $g$ Lefschetz fibration over $S^2$, the product of right handed Dehn twists $t_{\alpha_{i}}$ along the vanishing cycles $\alpha_i$, for $i = 1, \cdots , s$, determines the global monodromy of the Lefschetz fibration, the relation $t_{\alpha_1} \cdot t_{\alpha_2} \cdot ~ \cdots ~ \cdot t_{\alpha_s} = 1$ in $M_{g}$. Conversely, such a relation in $M_{g}$ determines a genus $g$ Lefschetz fibration over $S^2$ with the vanishing cycles $\alpha_1, \cdots, \alpha_s$.


\begin{lemma}
\label{lemma:pi1} $($\cite{GS}$)$
Let $f:X\to \mathbb{S}^2$ be a genus $g$ Lefschetz fibration with global monodromy given by the relation~$t_{\alpha_1} \cdot t_{\alpha_2} \cdot ~ \cdots ~ \cdot t_{\alpha_s} = 1$. Suppose that $f$ has a section. Then the fundamental group of $X$ is isomorphic to the fundamental group of $\Sigma_g$ divided out by the normal closure of the simple closed curves $\alpha_1,\alpha_2,\ldots,\alpha_s$, considered as elements in $\pi_1(\Sigma_g)$. In particular, there is an epimorphism $\pi_1(\Sigma_g)\to \pi_1(X)$
 
\end{lemma}

The following example is helpful for illustrating the discussion in the proof of our main Theorem.

\begin{example}

Let $\alpha_1$, $\alpha_2$, .... , $\alpha_{2g}$, $\alpha_{2g+1}$ denote the collection of simple closed curves given in Figure~\ref{fig:hyper}, and $c_{i}$ denote the right handed Dehn twists $t_{\alpha_i}$ along the curve $\alpha_i$. It is well-known that the following relations hold in the mapping class group $M_g$:  

\begin{equation}
\begin{array}{l}
\Gamma_{1}(g) = (c_1c_2 \cdots c_{2g-1}c_{2g}{c_{2g+1}}^2c_{2g}c_{2g-1} \cdots c_2c_1)^2 = 1.  \\
\end{array}
\end{equation}

\begin{figure}[ht]
\begin{center}
\includegraphics[scale=.43]{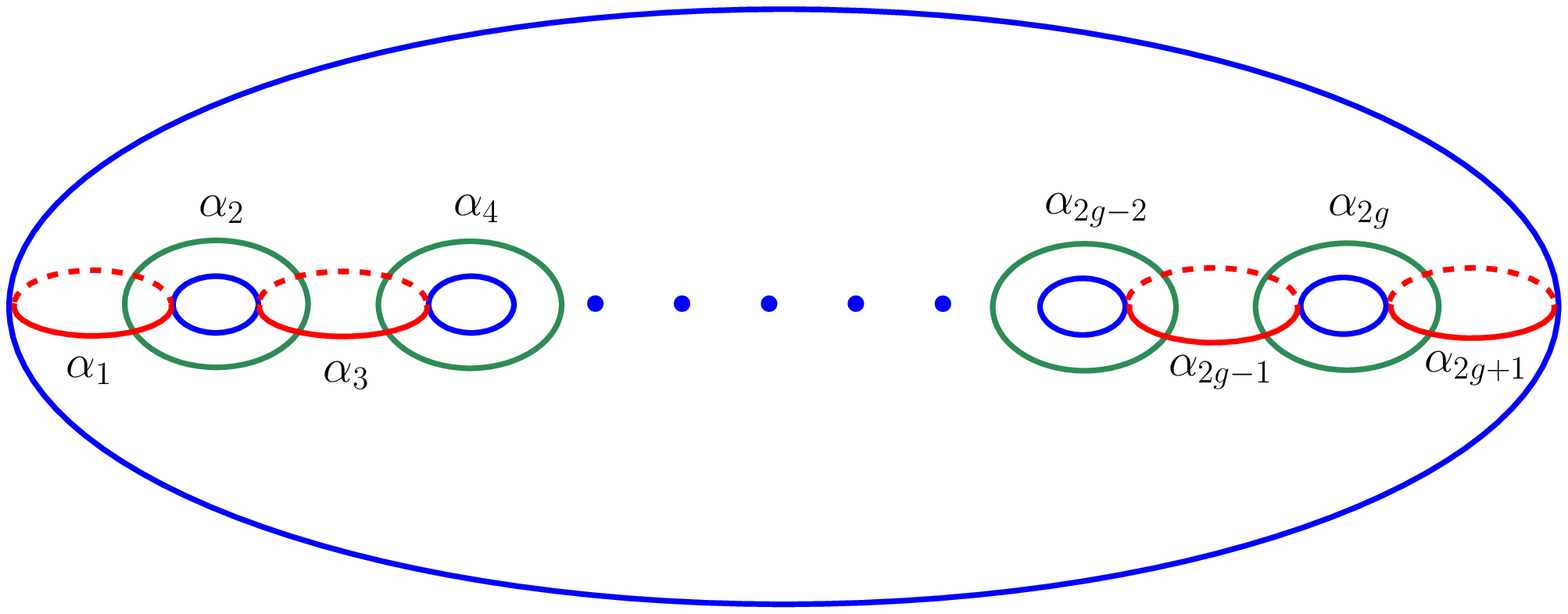}
\caption{Vanishing Cycles of the Genus $g$ Lefschetz Fibration on $X(g,1) = \CP\#(4g+5)\CPb$,}
\label{fig:hyper}
\end{center}
\end{figure}

\noindent The monodromy relation given above, corresponding to the genus $g$ Lefschetz fibration over $S^2$, has total space $X(g,1) = \CP\#(4g+5)\CPb$, the complex projective plane blown up at $4g+5$ points. Furthermore, it is well known that for $g \geq 2$, the above fibration on $X(g,1)$ admits $4g + 4$ disjoint $(-1)$-sphere sections (see \cite{T} for a proof of this fact using a mapping class group argument or \cite{ako} for a geometric argument).

\end{example}

\section{Symplectic connected sum and Luttinger surgery}

\subsection{Symplectic connected sum}

\begin{definition} Let $X_{1}$\/ and $X_{2}$\/ are closed, oriented, smooth $4$-manifolds, and $F_{i} \subset X_{i}$ are $2$-dimensional, smooth, closed, connected submanifolds in them. Suppose that  $[F_{1}]^{2} + [F_{2}]^{2} = 0$ and the genera of $F_{1}$ and $F_{2}$ are equal. We choose an orientation-preserving diffemorphism $\psi : F_{1} \longrightarrow  F_{2}$ and lift it to an orientation-reversing diffemorphism  $\Psi : \partial \nu F_{1} \longrightarrow  \partial \nu F_{2}$ between the boundaries of the tubular neighborhoods of $\nu F_{i}$. Using $\Psi$, we glue $X_{1} \setminus\nu F_{1}$ and $X_{2}\setminus \nu F_{2}$\/ along the boundary. This new oriented smooth $4$-manifold $X_{1}\#_{\Psi}X_{2}$ is called the \emph{connected sum}\/ of $X_{1}$\/ and $X_{2}$\/ along $F_{1}$ and $F_{2}$, determined by $\Psi$.

\end{definition}

\medskip

\begin{lemma} 
Let $X_{1}$ and $X_{2}$ be smooth\/ $4$-manifolds as above. Then 
\begin{eqnarray*}
c_{1}^{2}(X_{1}\#_{\Psi}X_{2}) &=&  c_{1}^{2}(X_{1}) + c_{1}^{2}(X_{2}) + 8(g-1),\\
\chi_{h}(X_{1}\#_{\Psi}X_{2}) &=&  \chi_{h}(X_{1}) + \chi_{h}(X_{2}) + (g-1),
\end{eqnarray*}
where $g$ is the genus of the surface $\Sigma$. 
\end{lemma}

\begin{proof}

The proofs of the above formulas simply are an easy consequence of the formulas

\begin{equation*}
e(X_{1}\#_{\Psi}X_{2})= e(X_{1}) + e(X_{2}) - 2e(\Sigma),\quad 
\sigma(X_{1}\#_{\Psi}X_{2}) = \sigma(X_{1}) + \sigma(X_{2}),    
\end{equation*} 

\noindent since $\chi_{h} = (\sigma  + e) / 4$ and  ${c_{1}^{2}} = 3\sigma + 2e$.

\end{proof}
 
\noindent If $X_{1}$, $X_{2}$ are symplectic manifolds and $F_{1}$, $F_{2}$ are symplectic submanifolds then according to theorem of Gompf \cite{G} $X_{1}\#_{\Psi}X_{2}$ admits a symplectic structure.  

\noindent We will use the following theorem  of M. Usher \cite{U} to show that the symplectic manifolds constructed in Sections 4 and 5 are minimal. 

\medskip

\begin{theorem}\cite{U} {\bf (Minimality of Sympletic Sums)} Let $Z = X_{1}\#_{F_{1} = F_{2}}X_{2}$ be sympletic fiber sum of manifolds $X_{1}$ and $X_{2}$. Then:

(i) If either $X_{1} \backslash F_{1}$ or $X_{2} \backslash F_{2}$ contains an embedded sympletic sphere of square $-1$, then $Z$ is not minimal.

(ii) If one of the summands $X_{i}$ (say $X_{1}$) admits the structure of an $S^{2}$-bundle over a surface of genus $g$ such that $F_{i}$ is a section of this fiber bundle, then $Z$ is minimal if and only if $X_{2}$ is minimal. 

(iii) In all other cases, $Z$ is minimal.

\end{theorem}




$$
$$



\subsection{Luttinger surgery}
\label{subsec:Luttinger}

In this subsection, we will briefly review a Luttinger surgery. For the details, we refer the reader to \cite{Lu} and \cite{ADK}. Luttinger surgery has been very effective tool recently for constructing exotic smooth structures on $4$-manifolds.

\begin{definition} Let $X$\/ be a symplectic $4$-manifold with a symplectic form $\omega$, and the torus $\Lambda$ be a Lagrangian submanifold of $X$ with self-intersection $0$. Given a simple loop $\lambda$ on $\Lambda$, let $\lambda'$ be a simple loop on $\partial(\nu\Lambda)$ that is parallel to $\lambda$ under the Lagrangian framing. For any integer $m$, the $(\Lambda,\lambda,1/m)$ \emph{Luttinger surgery}\/ on $X$\/ will be 
$X_{\Lambda,\lambda}(1/m) = ( X - \nu(\Lambda) ) \cup_{\phi} (S^1 \times S^1 \times D^2)$,  the $1/m$\/ surgery on $\Lambda$ with respect to $\lambda$ under the Lagrangian framing. Here 
$\phi : S^1 \times S^1 \times \partial D^2 \to \partial(X - \nu(\Lambda))$  denotes a gluing map satisfying $\phi([\partial D^2]) = m[{\lambda'}] + [\mu_{\Lambda}]$ in $H_{1}(\partial(X - \nu(\Lambda))$, where $\mu_{\Lambda}$ is a meridian of $\Lambda$.

\end{definition}

It is  shown in \cite{ADK} that $X_{\Lambda,\lambda}(1/m)$ possesses a symplectic form that restricts to the original symplectic form $\omega$ on $X\setminus\nu\Lambda$. The following lemma is easy to verify and the proof is left as an exercise to the reader.

\begin{lemma}
\begin{enumerate}
\noindent \item $\pi_1(X_{\Lambda,\lambda}(1/m)) = \pi_1(X- \Lambda)/N(\mu_{\Lambda} \lambda'^m)$.\\
\noindent \item $\sigma(X)=\sigma(X_{\Lambda,\lambda}(1/m))$ and $e(X)=e(X_{\Lambda,\lambda}(1/m))$.
\end{enumerate}
\end{lemma}



\section{Symplectic Building Blocks}\label{sbb}

In this section, we collect symplectic building blocks that are needed in our construction of exotic $4$-manifolds. Our building blocks will be the total space of the Lefschetz fibrations over $2$-sphere constructed by M. Korkmaz \cite{Ko} and Y. Gurtas \cite{Gu}, along with the symplectic building blocks that were constructed by the first author in \cite {A1}. The symplectic $4$-manifolds in \cite{A1} were obtained via knot surgery along the fibered knots, but they can be constructed via the sequence of Luttinger surgeries. In this paper we use Luttinger surgery, which will have some advantages for us in the computation of the fundamental groups.

\subsection{Korkmaz's fibration}

For the convenience of the reader, we provide necessary background and state main the results in \cite{Ko}. Let us consider the case when $g = 2k$. Recall that the four manifold $Y(k) = \Sigma_{k} \times \mathbb{S}^2\#4\CPb$ is the total space of the well known genus $g$ Lefschetz fibration over $\mathbb{S}^2$ \cite{M, Ko}. This was shown by Y. Matsumoto for $k = 1$, and in the case $k \geq 2$ by M. Korkmaz, by factorizing the \emph{vertical} involution $\theta$ of the genus $2k$ surface (See Figure~\ref{fig:inv}).  

\begin{figure}[ht]
\begin{center}
\includegraphics[scale=.27]{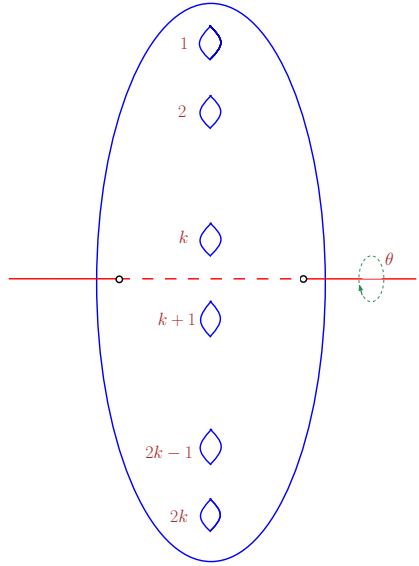}
\caption{The vertical involution $\theta$ of the genus $2k$ surface}
\label{fig:inv}
\end{center}
\end{figure}

The branched-cover description of the above Lefschetz fibrations can be given as follows: take a double branched cover of $\Sigma_{k} \times \mathbb{S}^2$ along the union of two disjoint copies of ${pt}\times \mathbb{S}^{2}$ and two disjoint copies of $\Sigma_{k} \times {pt}$ (See Figure~\ref{fig:relE}). The resulting branched cover has four singular points, corresponding to the number of the intersections points of the horizontal spheres and the vertical genus $k$ surfaces in the branch set. Next step is to desingularize this singular manifold to obtain $Y(k) = \Sigma_{k}\times S^2\#4\CPb$.  Observe that a generic fiber of the horizontal fibration is the double cover of $\mathbb{S}^2$, branched over two points. This gives a sphere fibration on $Y(k) = \Sigma_{k}\times \mathbb{S}^2\#4\CPb$. Also, a generic fiber of the vertical fibration is the double cover of $\Sigma_{k}$, branched over two points. Thus, a generic fiber is a genus $g$ surface. According to \cite{M, Ko}, each of the two singular fibers of the vertical fibration can be perturbed into $g+2$ Lefschetz type singular fibers. 

\begin{figure}[ht]
\begin{center}
\includegraphics[scale=.38]{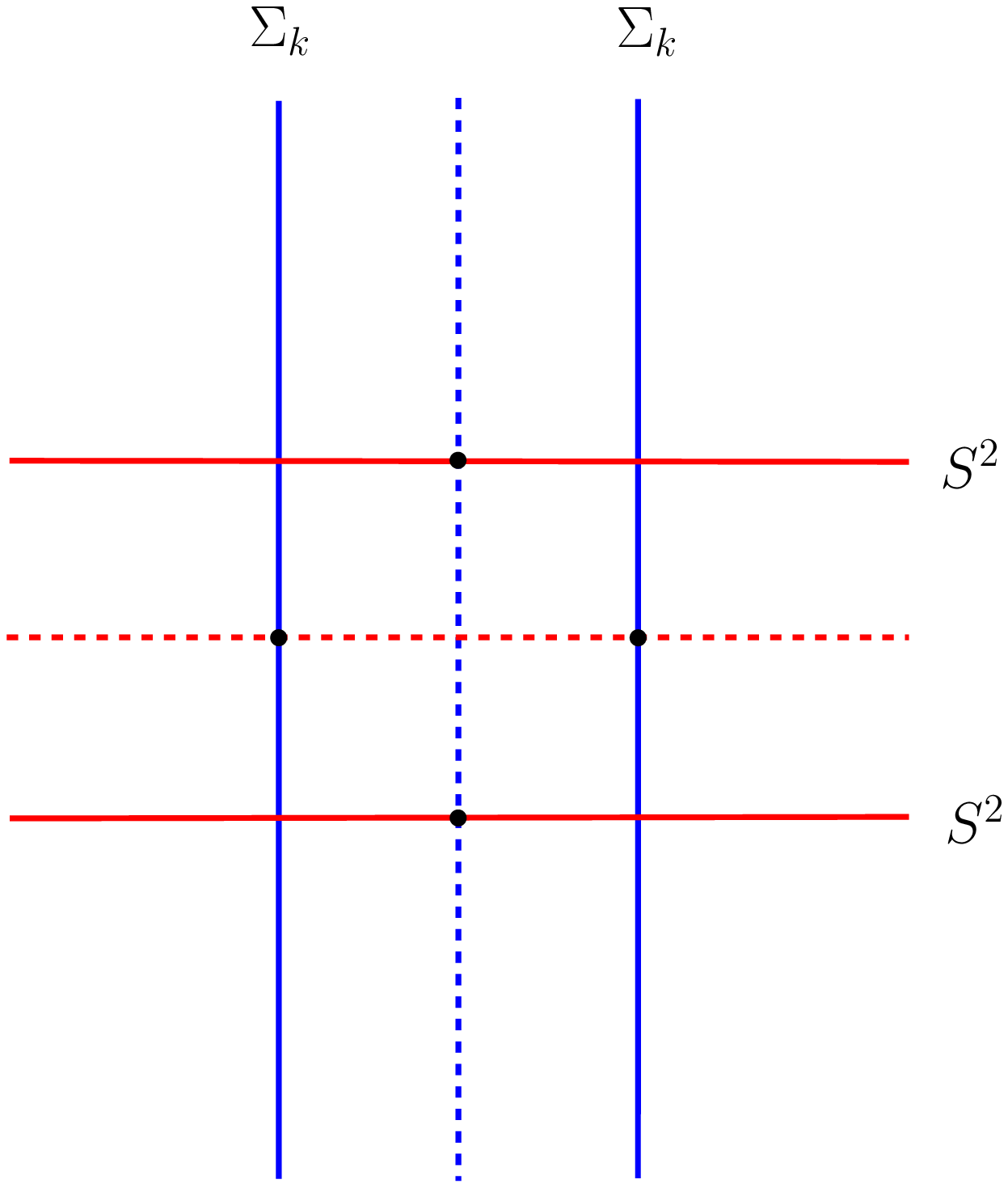}
\caption{The branch locus for $\Sigma_{k}\times S^2\#4\CPb$}
\label{fig:relE}
\end{center}
\end{figure}

Perhaps this is a good place to remark that when $g = 2k + 1$, the total space of the corresponding Lefschetz fibration is $\Sigma_{k} \times \mathbb{S}^2\#8\CPb$. In this case, the appropriate branched-cover description of the Lefschetz fibrations is given as follows: take a double branched cover of $\Sigma_{k} \times \mathbb{S}^2$ along the union of $4$ disjoint horizontal copies of ${pt}\times \mathbb{S}^{2}$ and two disjoint vertical copies of $\Sigma_{k} \times {pt}$. The resulting branched cover has $8$ singular points. By desingularize this manifold, we obtain $\Sigma_{k}\times S^2\#8\CPb$. A generic fiber of the vertical fibration is the double cover of $\Sigma_{k}$, branched over $4$ points. Thus, a generic fiber has genus $2k + 1$. See subsection~\ref{gfb} for more details, where this case occurs if we set $n = 2$ and $k \geq 1$.

In fact, the following theorem was proved in \cite{Ko}, which computes the global monodromy of the given Lefschetz fibration for both an even and an odd $g$.

\begin{theorem}\label{thm:theta} Let $\theta$ denote the vertical involution of the genus $g$ surface with $2$ fixed points. In the mapping class group $M_{g}$, the follwing relations between right Dehn twists hold: \ \

a) $(t_{B_{0}} t_{B_{1}}t_{B_{2}} \cdots t_{B_{g}}t_{c})^{2} = {\theta}^2 = 1$ if $g$ is even, \ \

b) $(t_{B_{0}} t_{B_{1}}t_{B_{2}} \cdots t_{B_{g}}(t_{a})^{2}(t_{b})^{2})^{2} = {\theta}^2 = 1$ if $g$ is odd. 

\end{theorem}

\noindent where $B_k$, $a$, $b$, $c$ are the simple closed curves defined as in Figure ~\ref{fig:relsE}.
 
\begin{figure}[ht]
\begin{center}
\includegraphics[scale=.43]{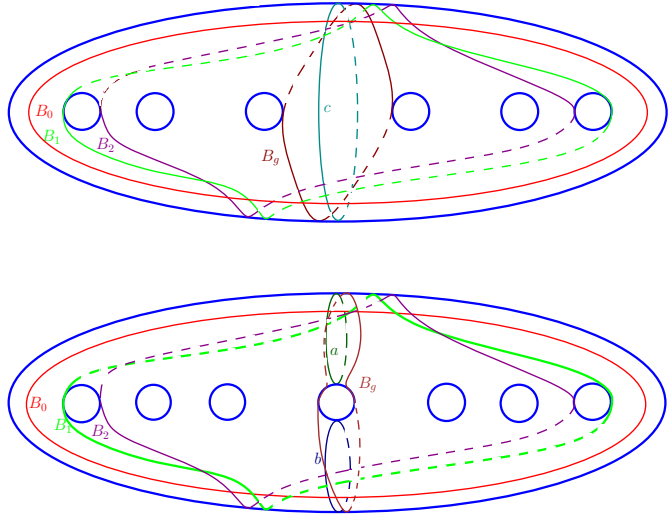}
\caption{The vanishing cycles}
\label{fig:relsE}
\end{center}
\end{figure}

Let us denote by $\Sigma_{2k}$ a regular fiber as the given Lefschetz fibration and the standard generators of fundamental group of $\Sigma_{2k}$ under the inclusion as $a_1, b_1, \cdots a_{2k}$ and $b_{2k}$. Using the homotopy exact sequence for a Lefschetz fibration, we have 

\begin {center} 

$\pi_{1}(\Sigma_{2k}) \longrightarrow \pi_{1}(Y(k)) \longrightarrow \pi_{1}(S^{2})$

\end{center}

\noindent According to \cite{Ko}, we have the following identification of the fundamental group of $Y(k)$: 

\medskip

$\pi_{1}(Y(k)) = \pi_{1}(\Sigma_{2k})/<B_{0},B_{1},\cdots,B_{g-1},B_{g},c> $ 

\medskip

\noindent It follows that the fundamental group of $Y(k)$ has a presentation with the generators $a_{1}$, $b_{1}$, $a_{2}$, $b_{2}$, $\cdots$, $a_{g}$, $b_{g}$ and the relations $[a_{1}, b_{1}][a_{2}, b_{2}] \cdots [a_{g}, b_{g}]=1$, $B_{0}=B_{1}=B_{2}= \cdots =B_{g}=c=1$. It was shown in \cite{Ko} the following identities hold

\noindent $B_{0} = b_{1}b_{2}\cdots b_{g}$ \

\noindent $B_{2i-1} = a_{i}b_{i}b_{i+1}\cdots b_{g+1-i}c_{g+1-i}a_{g+1-i}$,            $1\leq{i}\leq{k}$ \

\noindent $B_{2i} = a_{i}b_{i+1}b_{i+2} \cdots b_{g-i}c_{g-i}a_{g+1-i}$                 $1\leq{i}\leq{k-1}$ \

\noindent $B_{g} = B_{2k}=a_{k}c_{k}a_{k+1}$ \

\noindent $c = c_{r} = [a_{1}, b_{1}][a_{2}, b_{2}]\cdots [a_{r}, b_{r}]$ \

\medskip

Now, we prove a lemma which we use in the sequel.

\begin{lemma}\label{sixl}

The following relations hold in the fundamental group of $Y(k)$

\begin{gather}\label{Luttinger relations}
a_{1}a_{2k} = 1,\ \  a_2a_{2k-1}=1,\ \  \cdots ,\ \  a_{k}a_{k+1}=1,\\ \nonumber
b_{1}b_{2}\cdots b_{2k} = 1,\ \  b_2b_{3}\cdots b_{2k-1}= [a_{2k},b_{2k}],\ \  \cdots , \\ \nonumber
b_{i+1}b_{i+2} \cdots b_{g-i} = [a_{2k-i+1},b_{2k-i+1}] \cdots [a_{2k-1},b_{2k-1}][a_{2k},b_{2k}].\\ \nonumber
\end{gather}

\end{lemma}

\begin{proof} Using the relations $B_{0} = b_{1}b_{2}\cdots b_{g} = 1$, $B_{1} = a_{1}b_{1}b_{2}\cdots b_{g}c_{g}a_{g} = 1$, and $c_{g} = 1$ in the fundamental group of $Y(k)$, we easily see that $a_{1}a_{g} = a_{1}a_{2k} = 1$. Next, using the relations $B_{2}$ = $a_{1}b_{2}b_{3}\cdots b_{g-1}c_{g-1}a_{g} = 1$, $B_{3}$ = $a_{2}b_{2}b_{3}\cdots b_{g-1}c_{g-1}a_{g-1} = 1$, and $a_{1}a_{g} = 1$, we obtain $a_{2}a_{g-1}$ = $a_{2}a_{2k-1}$ = $1$. By continuing in this fashion (i.e. using the relations $B_{2i-2}$ = $a_{i-1}b_{i}b_{i+1} \cdots b_{g-i+1}c_{g-i+1}a_{g-i+2}$, $B_{2i-1}$ = $a_{i}b_{i}b_{i+1}\cdots b_{g-i+1}c_{g-i+1}a_{g+1-i}$ = $1$, and $a_{i-1}a_{g-i+2}=1$), we have $a_{i}a_{g-i+1}=1$ for any $i$ between $1$ and $k$. Furthemore, by considering the relations $B_{0}$ = $b_{1}b_{2}\cdots b_{g} = 1$, $c_{g-1} = [a_g, b_g]^{-1}$, and $a_{1}a_{g} = 1$, we have $1$ = $B_{2}$ = $a_{1}b_{2}b_{3}\cdots b_{g-1}c_{g-1}a_{g}$ = $b_{2}b_{3}\cdots b_{g-1}c_{g-1}$ = $b_{2}b_{3}\cdots b_{g-1} (b_{g}a_{g}{b_{g}}^{-1}{a_{g}}^{-1})$ = $1$. Consequently, using the relations $B_{2i}$ = $a_{i}b_{i+1}b_{i+2} \cdots b_{g-i}c_{g-i}a_{g+1-i} = 1$, and $a_{i}a_{g+i-1} = 1$, we have $b_{i+1}b_{i+2} \cdots b_{g-i} = [a_{g-i+1},b_{g-i+1}] \cdots [a_{g-1},b_{g-1}][a_g,b_g]$ for any $i$ between $1$ and $k$.          

\end{proof}

Using the Lemma above, we see that the fundamental group of $Y(k) = \Sigma_{k}\times \mathbb{S}^2\#4\CPb$ is isomorphic to the surface group $\Pi_k = \pi_1(\Sigma_k)$, generated by loops $a_1$, $b_{1}$, $\cdots$, $a_{k}$, and $b_{k}$. Furthemore, the fundamental group of the complement of $Y(k) \setminus \nu (\Sigma_{2k})$ is also $\Pi_k$. The normal circle $\mu = {pt} \times \partial (\mathbb{D}^{2})$ to $\Sigma_{2k}$ can be deformed using an exceptional sphere section, thus trivial in $\pi_{1}(Y(k) \setminus \nu \Sigma_{2k})$.


\subsection{Gurtas' fibration}\label{gfb}

In \cite{Gu}, Yusuf Gurtas generalized the constructions in \cite{M, Ko} even further. In \cite{Gu} he presented the positive Dehn twist expression for a new set of involutions in the mapping class group $M_{h+v}$ of a compact, closed, oriented $2$-dimensional surface $\Sigma_{h+v}$. These involutions were obtained by gluing the horizontal involution on a surface $\Sig_{h}$ and the vertical involution on a surface $\Sig_{v}$, where $v$ is a positive even number. Let $\theta$ denote the involution under consideration on the surface $\Sig_{h+v}$, as shown in Figure~\ref{fig:f4} below. According to Gurtas \cite{Gu}, $\theta$ can be expressed as a product of $8h + 2v + 4$ positive Dehn twists. Observe that if we set $h = 0$ and $v \geq 2$, we recover the family of Lefschetz fibrations in \cite{M, Ko}.

\begin{figure}[ht]
\begin{center}
\includegraphics[scale=.43]{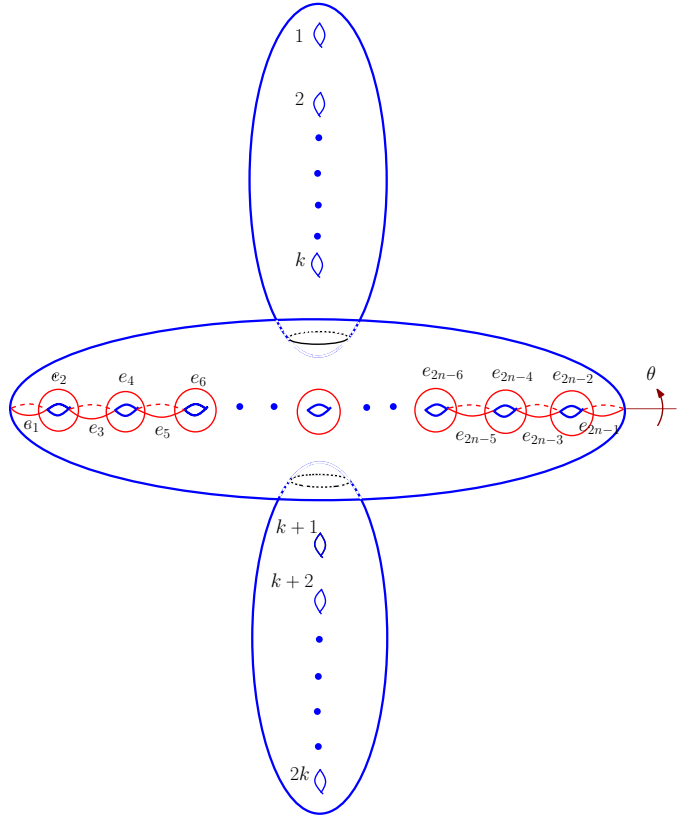}
\caption{The involution $\theta$ of the surface $\Sigma_{2k+n-1}$}
\label{fig:f4}
\end{center}
\end{figure}

In the above notation, let us set $n = h + 1$ and $v = 2k$, for reasons which will be clear soon. Let $Y(n,k)$ denote the total space of the Lefschetz fibration defined by the word $\theta^{2} = 1$ in the mapping class group $M_{2k+n-1}$. The manifold $Y(n,k)$ has a genus $g = 2k + n - 1$ Lefschetz fibration over $\mathbb{S}^{2}$ with $s = 8h + 2v + 4 = 8(n-1) + 2(2k) + 4 = 8n + 4k - 4$ singular fiberes and the vanishing cycles all are about nonseparating curves \cite{Gu}. The Euler characteristic of the symplectic $4$-manifold $Y(n,k)$ can be computed using the following formula $e(Y(n-1, k)) = e(\mathbb{S}^2)e(F) + s = 2(2-2(n+ 2k -1)) + 8n + 4k - 4 = 4n - 4k + 8$. The signature of the Lefschetz fibration described by the word $\theta^{2} = 1$ was computed in \cite{Gu}: $\sigma(Y(n,k)) = -4(n+2k-1)$ (see also related work in \cite{Yu}). Now, using the formulas $\chi_{h} = (\sigma  + e) / 4$,  ${c_{1}^{2}}= 3\sigma + 2e$, we compute: \ $\chi_{h}(Y(n,k)) = 1 - 3k$ \, and \ ${c_{1}^{2}}(Y(n,k)) = -4(n+4k-3)$.


It is not too difficult to determine what should be the appropriate branched-cover description of the above Lefschetz fibrations. It is given as follows: take a double branched cover of $\Sigma_{k} \times \mathbb{S}^2$ along the union of $2n$ disjoint copies of ${pt}\times \mathbb{S}^{2}$ and two disjoint copies of $\Sigma_{k} \times {pt}$ (See Figure~\ref{fig:gbc}). The resulting branched cover has $4n$ singular points, corresponding to the number of the intersections points of the horizontal spheres and the vertical genus $k$ surfaces in the branch set. We desingularize this manifold to obtain $Y(n,k) = \Sigma_{k}\times \mathbb{S}^2\#4n\CPb$. A generic horizontal fiber is the double cover of $\mathbb{S}^2$, branched over two points. Thus, we have a sphere fibration on $Y(n,k) = \Sigma_{k}\times \mathbb{S}^2\#4\CPb$. A generic fiber of the vertical fibration is the double cover of $\Sigma_{k}$, branched over $2n$ points. Thus, a generic fiber of the vertical fibration has genus $n + 2k - 1$. Furthermore, two complicated singular fibers of the vertical fibration can be perturbed into $4n + 2k - 2$ Lefschetz type singular fibers. 

\begin{figure}[ht]
\begin{center}
\includegraphics[scale=.43]{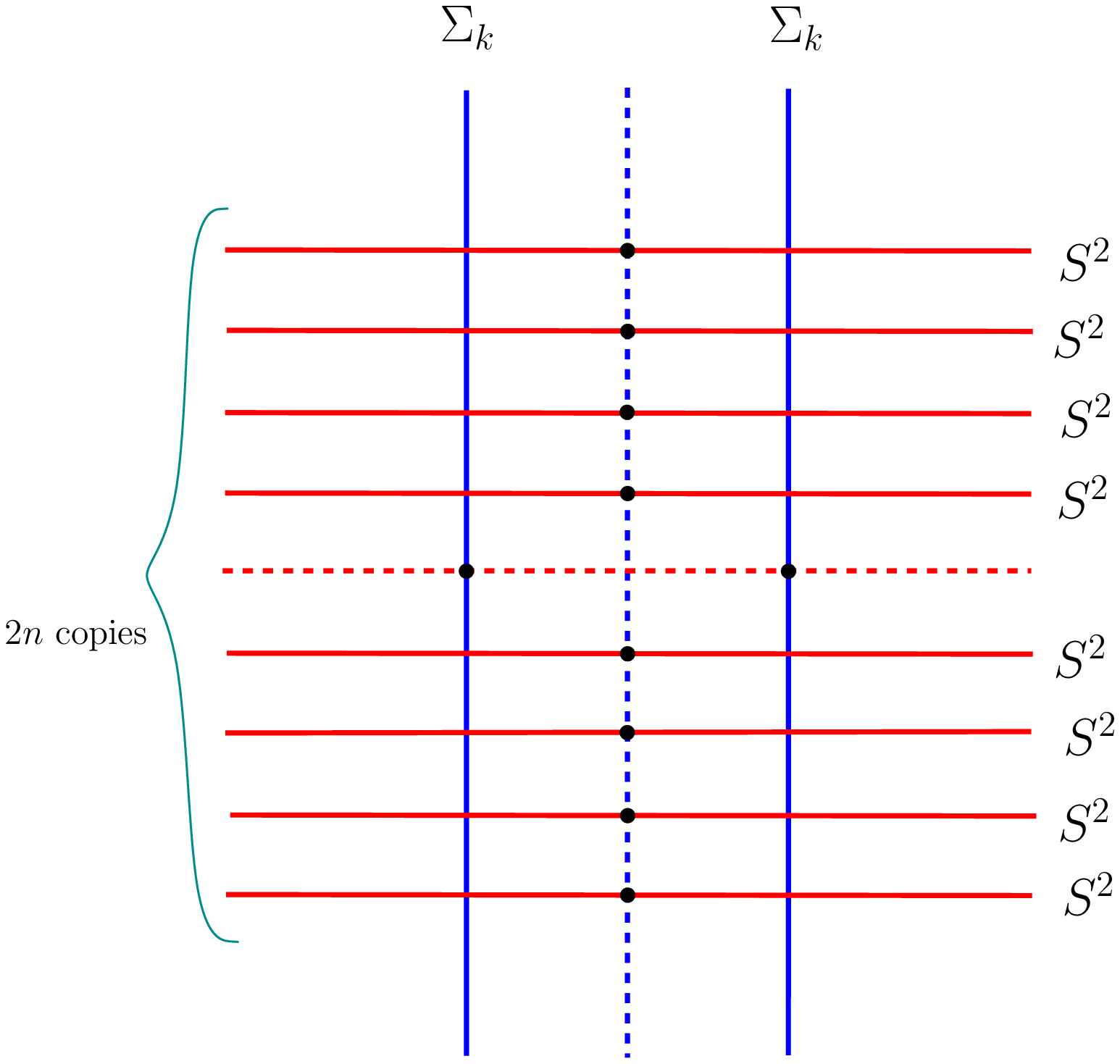}
\caption{The branch locus for $\Sigma_{k}\times S^2\#4n\CPb$}
\label{fig:gbc}
\end{center}
\end{figure}

Next, we recall the main theorem of \cite{Gu}, which will be needed in the proof of our Lemma~\ref{sixl}. We refer the reader to \cite{Gu} for any unexplained notation. 

\begin{theorem} \label{mainthm} The positive Dehn twist expression for the involution $\theta$ is given by
\[
\theta =e_{2i+2}\cdots e_{2n-2}e_{2h-1}e_{2i}\cdots
e_{2}e_{1}B_{0}e_{2n-1}e_{2n-2}\cdots e_{2i+2}e_{1}e_{2}\cdots
e_{2i}B_{1}B_{2}\cdots B_{4k-1}B_{4k}e_{2i+1}.
\]
\end{theorem}

Now, we prove two lemmas which we use in the sequel. 

\begin{lemma}\label{sixl}

The following relations hold in the fundamental group of $Y(n,k)$

\begin{gather}\label{Luttinger relations}
e_{1} = 1,\ \  e_2 =1,\ \  \cdots ,\ \  e_{2n-2} =1, \ \  e_{2n-1} = 1 \\ \nonumber
a_{1}a_{2k} = 1,\ \  a_2a_{2k-1}=1,\ \  \cdots ,\ \  a_{k}a_{k+1}=1,\\ \nonumber
\end{gather}

\end{lemma}

\begin{proof} Notice that the first set of relations simply follows from the fact that the Dehn twists along the curves $e_1$, $e_2$, $\cdots$, $e_{2n-1}$ appear in the factorization of $\theta$. To prove the second set of relations, we use the relations $B_{0} = b_{1}b_{2}\cdots b_{2k} = 1$, $B_{1} = a_{1}b_{1}b_{2}\cdots b_{2k}c_{2k}a_{2k} = 1$, $\cdots$, $B_{2i-1} = a_{i}b_{i}b_{i+1} \cdots b_{2k+1-i}c_{2k+1-i}a_{2k+1-i}$, $B_{2i} = a_{i}b_{i+1}b_{i+2} \cdots b_{2k-i}c_{2k-i}a_{2k+1-i}$, $\cdots$, $B_{2k} = a_{k}c_{k}a_{k+1}$, and $c_{k} = 1$ in the fundamental group of $Y(n,k)$. The proof is identical to the proof of Lemma~\ref{sixl} and therefore is omitted.

\end{proof}

\begin{lemma}\label{sixl}

The genus $2g+n-1$ Lefschetz fibration on $Y(n,k)$ admits at least $4n$ disjoint $-1$ sphere sections. 

\end{lemma}

\begin{proof} We first observe that $Y(n,k)$ is the symplectic sum $Y(k) = \Sigma_{k}\times \mathbb{S}^2$ and $X(n,1) = \CP\#(4n+5)\CPb$ along the spheres $\pt \times \mathbb{S}^2$ and the sphere fiber of $X(n,1)$. Since a generic sphere fiber of $X(n,1)$ intersects a generic genus $n-1$ fiber at two points, after the symplectic sum we obtain a genus $2k + n - 1$ fibration on $Y(n,k)$. Since for $n \geq 2$, the genus $n-1$ fibration on $X(n,1)$ admits $4n$ disjoint $(-1)$-sphere sections (see \cite{T}), these $-1$ sphere sections extends to $Y(n,k)$.          

\end{proof}

\subsection{Luttinger surgeries on product manifolds $\Sigma_{n}\times \Sigma_{2}$ and $\Sigma_{n}\times \mathbb{T}^2$}\label{L}

The following two family of symplectic building blocks will be used in our construction. They are obtained from $\Sigma_{n}\times \Sigma_{2}$ and $\Sigma_{n}\times \mathbb{T}^2$ by performing a sequence of Luttinger surgeries along the Lagrangian tori \cite{FPS, AP1, AP2}. The first family has $b_1 = 0$, and the second has $b_1 = 2$. One can relate these symplectic building blocks $X_{K}$, $V_{KK'}$ and $Y_{K}$, $W_{KK'}$ in \cite{A1, A2}, where $K$ and $K'$ are genus $1$ and genus $n-1$ fibered knots respectively. In order to see this connection, one should view the knot surgery manifold $M_{K} \times \mathbb{S}^1$ via the sequence of $2g$ Luttinger surgeries on $\Sigma_{g} \times \mathbb{T}^2$, where $K$ is a fibered genus $g$ knot in $\mathbb{S}^3$. This was carefully explained in \cite{ABP} for a trefoil knot. We also refer the reader \cite{Lu} (pages 225-226). 

To construct the first family of examples, we proceed as follows. Let us fix integers $n\geq 2$, $p_{i} \geq 0$ and $q_{i} \geq 0$ , where $1 \leq i \leq n$. Let $Y_{n}(1/p_{1},1/q_{1}, \cdots, 1/p_{n}, 1/q_{n})$ denote symplectic $4$-manifold obtained by performing the following $2n + 4$ Luttinger surgeries on $\Sigma_{n}\times \Sigma_{2}$. These $2n+4$ surgeries comprise of the following $8$ surgeries

\begin{eqnarray}\label{first 8 Luttinger surgeries}
&&(a_1' \times c_1', a_1', -1), \ \ \nonumber (b_1' \times c_1'', b_1', -1), \\ \nonumber
&&(a_2' \times c_2', a_2', -1), \ \ (b_2' \times c_2'', b_2', -1),\\ \nonumber
&&(a_2' \times c_1', c_1', +1/p_1), \ \ (a_2'' \times d_1', d_1', +1/q_1),\\ \nonumber
&&(a_1' \times c_2', c_2', +1/p_2), \ \ (a_1'' \times d_2', d_2', +1/q_2),
\end{eqnarray}
together with the following $2(n-2)$ additional Luttinger surgeries
\begin{gather*}
(b_1'\times c_3', c_3',  -1/p_3), \ \ 
(b_2'\times d_3', d_3', -1/q_3), \\  
\dots  , \ \ \dots, \\
(b_1'\times c_n', c_n',  -1/p_n), \ \
(b_2'\times d_n', d_n', -1/q_n).
\end{gather*}
Here, $a_i,b_i$ ($i=1,2$) and $c_j,d_j$ ($j=1,\dots,n$) denote the standard loops that generate $\pi_1(\Sigma_2)$ and $\pi_1(\Sigma_n)$, respectively. See Figure~\ref{fig:lagrangian-pair} for a typical Lagrangian tori along which the Luttinger surgeries are performed.  

\begin{figure}[ht]
\begin{center}
\includegraphics[scale=.49]{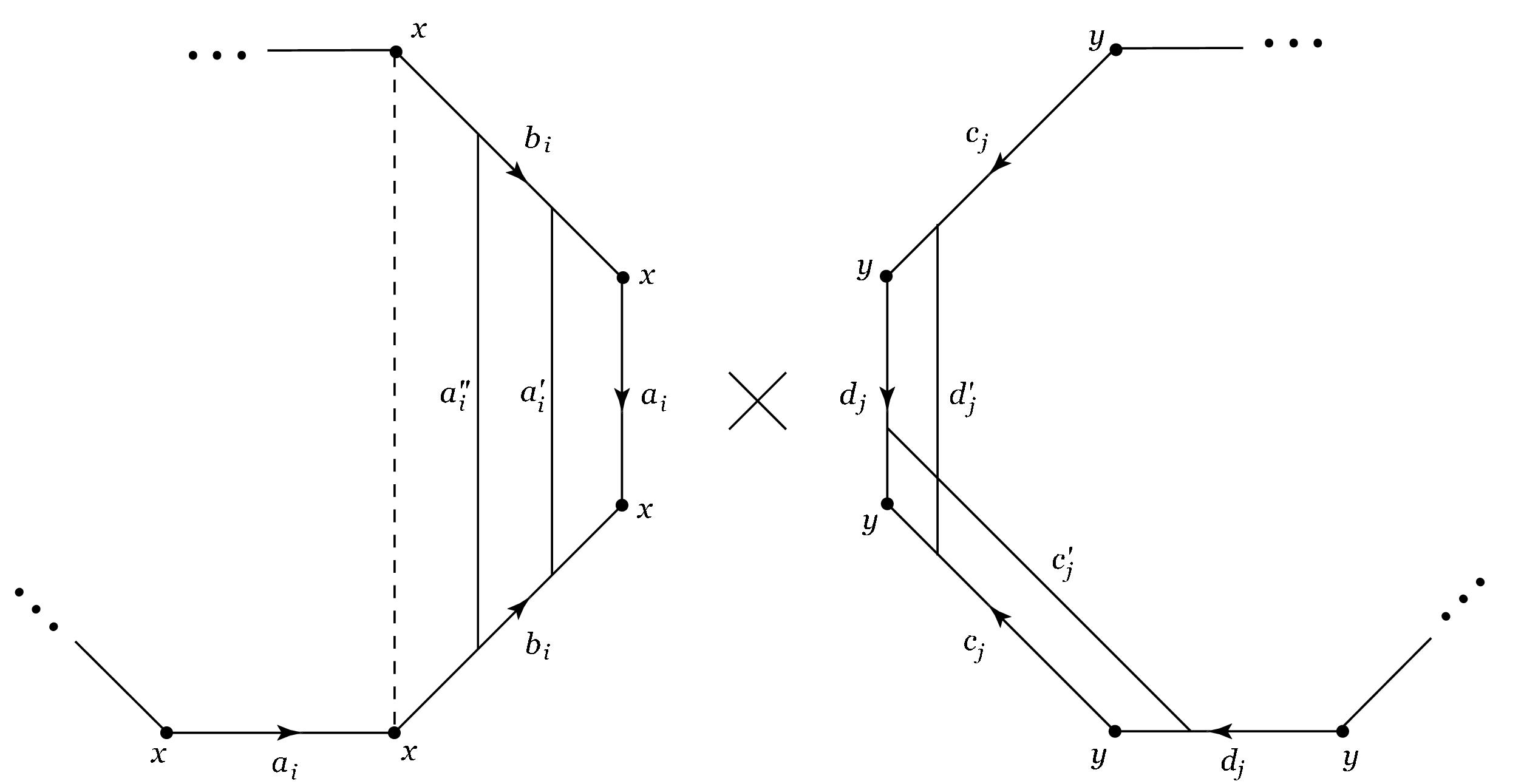}
\caption{Lagrangian tori $a_i'\times c_j'$ and $a_i''\times d_j'$}
\label{fig:lagrangian-pair}
\end{center}
\end{figure}

The Euler characteristic of $Y_{n}(1/p_{1},1/q_{1}, \cdots, 1/p_{n}, 1/q_{n})$ is $4n-4$ and its signature is $0$. The fundamental group $\pi_1(Y_{n}(1/p_{1},1/q_{1}, \cdots, 1/p_{n}, 1/q_{n}))$ is generated by $a_i,b_i,c_j,d_j$ ($i=1,2$ and $j=1,\dots,n$) and the following relations hold in $\pi_1(Y_{n}(1/p_{1},1/q_{1}, \cdots, 1/p_{n}, 1/q_{n}))$:  

\begin{gather}\label{Luttinger relations}
[b_1^{-1},d_1^{-1}]=a_1,\ \  [a_1^{-1},d_1]=b_1,\ \  [b_2^{-1},d_2^{-1}]=a_2,\ \  [a_2^{-1},d_2]=b_2,\\ \nonumber
[d_1^{-1},b_2^{-1}]=c_1^{p_1},\ \ [c_1^{-1},b_2]=d_1^{q_1},\ \ [d^{-1}_2,b^{-1}_1]=c_2^{p_2},\ \ [c_2^{-1},b_1]=d_2^{q_2},\\ \nonumber
 [a_1,c_1]=1, \ \ [a_1,c_2]=1,\ \  [a_1,d_2]=1,\ \ [b_1,c_1]=1,\\ \nonumber
[a_2,c_1]=1, \ \ [a_2,c_2]=1,\ \  [a_2,d_1]=1,\ \ [b_2,c_2]=1,\\ \nonumber
[a_1,b_1][a_2,b_2]=1,\ \ \prod_{j=1}^n[c_j,d_j]=1,\\ \nonumber
[a_1^{-1},d_3^{-1}]=c_3^{p_3}, \ \ [a_2^{-1} ,c_3^{-1}] =d_3^{q_3}, \  \dots, \ 
[a_1^{-1},d_n^{-1}]=c_n^{p_n}, \ \ [a_2^{-1} ,c_n^{-1}] =d_n^{q_n},\\ \nonumber
[b_1,c_3]=1,\ \  [b_2,d_3]=1,\ \dots, \
[b_1,c_n]=1,\ \ [b_2,d_n]=1.
\end{gather}

The surfaces $\Sigma_2\times\{{\rm pt}\}$ and $\{{\rm pt}\}\times \Sigma_n$ in $\Sigma_2\times\Sigma_n$ descend to surfaces in $Y_{n}(1/p_{1},1/q_{1}, \cdots, 1/p_{n}, 1/q_{n})$. 
They are symplectic submanifolds in \\
$Y_{n}(1/p_{1},1/q_{1}, \cdots, 1/p_{n}, 1/q_{n})$. We will denote their images by $\Sigma_2$ and $\Sigma_n$. Note that $[\Sigma_2]^2=[\Sigma_n]^2=0$ and $[\Sigma_2]\cdot[\Sigma_n]=1$.

Now, we consider a slightly different family. Let us fix a quadruple of integers $n \geq 2$, $m \geq 1$, $p \geq 1$ and $q \geq 1$. Let $Y_{n}(1/p,m/q)$ denote smooth $4$-manifold obtained by performing the following 
$2n$ torus surgeries on $\Sigma_n\times \mathbb{T}^2$:

\begin{eqnarray}\label{eq: Luttinger surgeries for Y_1(m)} 
&&(a_1' \times c', a_1', -1), \ \ (b_1' \times c'', b_1', -1),\\  \nonumber
&&(a_2' \times c', a_2', -1), \ \ (b_2' \times c'', b_2', -1),\\  \nonumber
&& \cdots, \ \ \cdots \\ \nonumber
&&(a_{n-1}' \times c', a_{n-1}', -1), \ \ (b_{n-1}' \times c'', b_{n-1}', -1),\\  \nonumber
&&(a_{n}' \times c', c', +1/p), \ \ (a_{n}'' \times d', d', +m/q).
\end{eqnarray}

Here, $a_i,b_i$ ($i=1,2, \cdots, m$) and $c,d$\/ denote the standard generators of $\pi_1(\Sigma_{2m})$ and $\pi_1(\mathbb{T}^2)$, respectively. Since all the torus surgeries above are Luttinger surgeries when $m = 1$, $Y_{n}(1/p,1/q)$ is a minimal symplectic 4-manifold. The fundamental group of $Y_{n}(1/p,m/q)$ is generated by $a_i,b_i$ ($i=1,2,3 \cdots, n$) and $c,d$, and the following relations hold in $\pi_1(Y_{n}(1/p,m/q))$:

\begin{gather}\label{Luttinger relations for Y_1(m)}
[b_1^{-1},d^{-1}]=a_1,\ \  [a_1^{-1},d]=b_1,\ \
[b_2^{-1},d^{-1}]=a_2,\ \  [a_2^{-1},d]=b_2,\\ \nonumber
\cdots,  \ \  \cdots,  \\ \nonumber
[b_{n-1}^{-1},d^{-1}]=a_{n-1},\ \  [a_{n-1}^{-1},d]=b_{n-1},\ \
[d^{-1},b_{n}^{-1}]=c^p,\ \ {[c^{-1},b_{n}]}^{-m}=d^q,\\ \nonumber
[a_1,c]=1,\ \  [b_1,c]=1,\ \ [a_2,c]=1,\ \  [b_2,c]=1,\\ \nonumber
[a_3,c]=1,\ \  [b_3,c]=1,\\ \nonumber
\cdots,  \ \  \cdots,  \\ \nonumber
[a_{n-1},c]=1,\ \  [b_{n-1},c]=1,\\ \nonumber
[a_{n},c]=1,\ \  [a_{n},d]=1,\\ \nonumber
[a_1,b_1][a_2,b_2] \cdots [a_n,b_n]=1,\ \ [c,d]=1.
\end{gather}

Let $\Sigma'_n \subset Y_{n}(1/p,l/q)$ be a genus $n$ surface that desend from the surface $\Sigma_{n}\times\{{\rm pt}\}$ in $\Sigma_{n}\times \mathbb{T}^2$.

\section{Construction of exotic $4$-manifolds}

In this section we prove the following theorem:

\begin{theorem}\label{thm:main}
Let $M$\/ be \/ $(2n+2k-3)\CP\#(6n+2k-3)\CPb$ for any $n \geq 1$ and $k \geq 1$. There exists a new family of smooth closed simply-connected minimal symplectic\/ $4$-manifold and an infinite family of non-symplectic $4$-manifolds that is homeomorphic but not diffeomorphic to\/ $M$.  
\end{theorem}

We will accomplish the proof of this theorem in several steps. First, we provide the construction which uses a combination of the Luttinger surgery \cite{Lu} and symplectic fiber sum operations \cite{G}. Since the order of these operations do not influence the outcome, our construction can be obtained by performing Luttinger surgeries on Lefschetz fibrations over higher genus surfaces. Next, we show that the fundamental groups are trivial, and determine the homeomorphism types of our examples. In the final step, we study the diffeomorphism types of our manifolds and distinguish them from $(2n+2k-3)\CP\#(6n+2k-3)\CPb$ by computing their Seiberg-Witten invariants and applying Theorem \cite{U}.  

Our first building block will be the symplectic $4$-manifold $Y(n,k) = \Sigma_{k} \times \mathbb{S}^2\#4n\CPb$ with a genus $2k+n-1$ symplectic submanifold $\Sigma_{2k+n-1}\subset Y(n,k)$, a regular fiber of the Lefschetz fibration which we discussed Section~\ref{sbb}. Here we endowed $Y(n,k) = \Sigma_{k} \times \mathbb{S}^2\#4n\CPb$ with the symplectic structure induced from the given Lefschetz fibration. The other building block will be the symplectic $4$-manifold $Y_{g}(1,1)$ along the symplectic submanifold $\Sigma'_g$ (see Section~\ref{sbb}), where we set $g = 2k + n -1$ and $p = q = m = 1$. We denote by $X(n,k)$ the symplectic $4$-manifold obtained by forming the symplectic fiber sum of $Y(n,k)$ and $Y_{g}(1,1)$\/ along the surfaces $\Sigma_{2k+n-1}$ and $\Sigma'_{g}$. 

Recall from Section~\ref{sbb} that loops $a_{1}$, $b_{1}$, $\cdots$, $a_{2k}$, and $b_{2k}$ generate the inclusion-induced image of $\pi_1(\Sigma_{2k+n-1}\times S^{1})$ inside $\pi_1(Y(n,k)\setminus\nu\Sigma_{2k+n-1})$. This is due to the fact that the loops $e_{1}$, $e_{2}$, $\cdots$, $e_{2n-2}$, $e_{2n-1}$ and the normal circle to $\mu=\{{\rm pt}\}\times S^1$ to $\Sigma_{2k+n-1}$ are all nullhomotopic in $(Y(n,k)\setminus\nu\Sigma_{2k+n-1})$. As before, let $a_{1}'$, $b_{1}'$, $\cdots$ $a_{g}'$, $b_{g}'$ and $\mu' = [c,d]$ generate $\pi_1(\Sigma'_{g}\times S^{1})$ in $\pi_1(Y_{g}(1,1)\setminus\nu\Sigma'_{g})$. 

We choose the gluing diffeomorphism  $\psi : \Sigma_{2g+n-1}\times S^1 \rightarrow \Sigma'_{g}\times S^1$ that maps the generators of the fundamental groups as follows: 

\begin{align*}
\psi_{\ast} (a_{1}) = a_{1}',\; 
\psi_{\ast} (b_{1}) = b_{1}',\; 
\psi_{\ast} (a_{2}) = a_{2}',\; 
\psi_{\ast} (b_{2}) = b_{2}',\; 
\cdots, 
\psi_{\ast} (a_{2k}) = a_{2k}',\;
\psi_{\ast} (b_{2k}) = b_{2k}',\;
\\
\psi_{\ast} (e_{1}) = a_{2k+1}',\;
\psi_{\ast} (e_{2}) = b_{2k+1}',\;
\cdots, 
\psi_{\ast} (e_{2n-3}) = a_{2k+n-1}',\;
\psi_{\ast} (e_{2n-2}) = b_{2k+n-1}',\;
\psi_{\ast} (\mu) = \mu'.
\end{align*} 

By Gompf's theorem in \cite{G}, $X(n,k)$ is symplectic. 

\begin{lemma}\label{lemma:pi_1(X)=1} 
$X(n,k)$ is simply-connected. 
\end{lemma}

\begin{proof}
By Van Kampen's theorem, we have 

\begin{equation*}
\pi_{1}(X(n,k)) \:=\: \frac{\pi_{1}(Y(n,k)\setminus \nu \Sigma_{2k +n -1})\ast \pi_{1}(Y_{g}(1,1) \setminus \nu \Sigma'_{g})}{\langle a_{1} = a_{1}',\, b_{1} = b_{1}',\, \cdots,\, a_{2k} = a_{2k}',\, b_{2k} = b_{2k}',\, 
\cdots,\, e_{2n-2} = b_{2k+n-1}',\, \mu = \mu' = 1 \rangle}.
\end{equation*}

\noindent Since the loops $e_{1}$, $e_{2}$, $\cdots$, $e_{2n-3}$, $e_{2n-2}$ and the normal circle to $\mu=\{{\rm pt}\}\times S^1$ are all nullhomotopic in $Y(n,k)\setminus\nu\Sigma_{2k+n-1}$, we get the following presentation for the fundamental group of $X(n,k)$. 

\begin{eqnarray}
\pi_1(X(n,k))  \nonumber &=& \langle 
a_{1}, b_{1}, \cdots, a_{2k}, b_{2k};\, c, d ;\, \mid \label{eq:pi_1(X(n,k))} \\ \nonumber
&&[b_1^{-1},d^{-1}]=a_1,\ \  [a_1^{-1},d]=b_1, \\ \nonumber
&&\cdots,  \ \ \cdots,  \\ \nonumber
&&[b_{2k}^{-1},d^{-1}]=a_{2k},\ \  [a_{2k}^{-1},d]=b_{2k},\ \
[d^{-1},b_{2k}^{-1}]=c,\ \ [c^{-1},b_{2k}]=d,\\ \nonumber
&&[a_1,c]=1,\ \  [b_1,c]=1,\ \ [a_2,c]=1,\ \  [b_2,c]=1,\\ \nonumber
&&[a_3,c]=1,\ \  [b_3,c]=1,\\ \nonumber
&&\cdots,  \ \ \cdots,  \\ \nonumber
&&[a_{2k-1},c]=1,\ \  [b_{2k-1},c]=1,\\ \nonumber
&&[a_{2k},c]=1,\ \  [a_{2k},d]=1,\\ \nonumber
&&[a_1,b_1][a_2,b_2] \cdots [a_{2k},b_{2k}]=1,\ \ [c,d]=1, \\ \nonumber
&&a_{1}a_{2k} = 1,\ \  a_2a_{2k-1}=1,\ \  \cdots ,\ \  a_{k}a_{k+1}=1, \\ \nonumber
&&b_{1}b_{2}\cdots b_{2k} = 1,\ \  b_{2}b_{3}\cdots b_{2k-1}= [a_{2k},b_{2k}], \\ \nonumber
&&\cdots , \cdots, \cdots \\ \nonumber
&&b_{i+1}b_{i+2} \cdots b_{2k-i} = [a_{2k-i+1},b_{2k-i+1}] \cdots [a_{2k-1},b_{2k-1}][a_{2k},b_{2k}], \\ \nonumber 
&&[a_1,b_1][a_2,b_2] \cdots [a_{k},b_{k}]=1  \rangle.
\end{eqnarray}

To prove $\pi_1(X(n,k))=1$, it is enough to prove that $b_{1} = 1$, which in turn will imply that all other generators are trivial. Using the last set of identities, we have ${a_{1}}^{-1}= a_{2k}$, $\cdots$ , ${a_{k}}^{-1} = a_{k+1}$. Now, we rewrite the relation $[a_{1}^{-1},d]=b_{1}$ as\/ $[a_{2k},d] = b_{1}$. Since $[a_{2k},d] = 1$, we obtain $b_{1} = 1$, which in turn imply that $a_{1} = a_{2k} = b_{2k} = c = d = 1$ using the relations $[b_1^{-1},d^{-1}]=a_1$, $a_{2k} = {a_{1}}^{-1}$, $b_{1}b_{2}\cdots b_{2k} = 1$, $b_{2}b_{3}\cdots b_{2k-1}= [a_{2k},b_{2k}]$, $[d^{-1},b_{2k}^{-1}]=c$ and $[c^{-1},b_{2k}]=d$. Since \/ $[b_{i}^{-1},d^{-1}]=a_{i}$ and\/ $[a_{i}^{-1},d]=b_{i}$, for any $1 \leq i \leq 2k-1$, we have $a_{2} = \cdots = a_{2k-1} = b_{2} = \cdots = b_{2k-1} = 1$. Thus, $\pi_1(X(n,k))$ is trivial.

\end{proof}

\begin{lemma} 
$e(X(n,k))= 8n + 4k - 4$, $\sigma (X(n,k)) = - 4n$, $c_1^{2}(X(n,k)) = 4n + 8k - 8$, and\/ $\chi_h(X(n,k)) = n + k - 1$.
\end{lemma}

\begin{proof} 
Since \/ $e(X(n,k))=e(Y(n,k))+e(Y_{g}(1,1))+4(2k+n-2)$, 
$\sigma (X(n,k)) = \sigma(Y(n,k)) + \sigma(Y_{g}(1,1))$, 
$c_1^{2}(X(n,k)) = c_1^{2}(Y(n,k)) + c_1^{2}(Y_{g}(1,1)) + 8(2k+n-2)$, and 
$\chi_{h}(X(n,k)) = \chi_{h}(Y(n,k)) + \chi_{h}(Y_{g}(1,1)) + (2k+n-2)$.  
We easily compute that $e(Y_{g}(1,1))=0$, 
$\sigma(Y_{g}(1,1))=0$, 
 $c_1^{2}(Y_{g}(1,1)=0$, 
and $\chi_{h}(Y_{g}(1,1))=0$.
Since $e(Y(n,k))=4 - 4k + 4n$, 
$\sigma(Y(n,k)) = -4n$, 
$c_1^{2}(Y(n,k)) = 8 - 8k - 4n$ 
and $\chi_{h} (Y(n,k)) = 1 - k$, which proves the lemma. 
\end{proof}

Using the Freedman's classification theorem for simply-connected 4-manifolds (cf.\ \cite{F}), we conclude that $X(n,k)$\/ is homeomorphic to $(2n+2k-3)\CP\#(6n+2k-3)\CPb$. Since $X(n,k)$ is symplectic, by Taubes's theorem (cf.\ \cite{Ta}) ${\rm SW}_{X(n,k)}( K_{X(n,k)} ) = \pm 1$, where $K_{X(n,k)}$ is the canonical class of $X(n,k)$. Next, by we applying the connected sum theorem (cf.\ \cite{W}) for the Seiberg-Witten invariant, we deduce that the Seiberg-Witten invariant of $(2n+2k-3)\CP\#(6n+2k-3)\CPb$ is trivial. Since the Seiberg-Witten invariant is a diffeomorphism invariant, we conclude that $X(n,k)$\/ is not diffeomorphic to $(2n+2k-3)\CP\#(6n+2k-3)\CPb$. 

All $4n$ exceptional spheres $E_{1}$, $E_{2}$, $\cdots$, $E_{4n-1}$, and $E_{4n}$, which are the sections of the genus $2k + n - 1$ fibration on $Y(n,k) = \Sigma_{k} \times \mathbb{S}^2\#4n\CPb$, meet with the fiber $\Sigma = 2n(\Sigma_{k}\times \{{\rm pt}\}) + \{{\rm pt}\} \times \mathbb{S}^{2} - E_{1} - E_{2} - \cdots - E_{4n-1} - E_{4n}$. Furthemore, any embedded symplectic $-1$ sphere in $Y(n,k)$\/ has the form $r\mathbb{S}^{2} \pm E_{i}$, thus intersect non-trivially with the fiber $\Sigma = 2n(\Sigma_{k}\times \{{\rm pt}\}) + \{{\rm pt}\} \times \mathbb{S}^{2} - E_{1} - E_{2} - \cdots - E_{4n-1} - E_{4n}$. Using the Usher's Theorem, the symplectic sum manifold $X(n,k)$\/ is a minimal symplectic $4$-manifold. Since symplectic minimality implies smooth minimality (cf.\ \cite{Li}), $X(n,k)$\/ is also smoothly minimal.

In order to produce an infinite family of exotic $(2n+2k-3)\CP\#(6n+2k-3)\CPb$'s, we replace the building block $Y_{g}(1,1)$\/ with $Y_{g}(1,m)$, where $|m| > 1$. Let us denote the resulting smooth manifolds as $X(n,k,m)$\/. In the presentation of the fundamental group, this amounts replacing a single relation $[c^{-1},b_{n}] =d$ with ${[c^{-1},b_{n}]}^{-m}=d$. Notice that chainging this relation does not affect our calculation of $\pi_{1}(X(n,k)) = 1$. Hence all the fundamental group calculations follow the same lines and result in the trivial group.

Moreover, using the argument similar as in \cite{ABP} (Section 4, pages 12-18), we see that $X(n,k)$ has one basic class up to sign, the canonical class $\pm K_{X(n,k)}$. Let $X(n,k)_{0}$ denote the symplectic $4$-manifold obtained by performing the following Luttinger surgery on: $(a_{n}'' \times d', d', 0/1)$. It is easy to see that $\pi_{1}(X(n,k)_{0}) = \mathbb{Z}$ and the canonical class $K_{X(n,k)_{0})} = 2[\Sigma_l]+ \sum_{j=1}^{4k}[R_j]$, where $R_j$ are tori of self-intersection $-1$. Furthermore, the basic class $\beta_{n,k,m}$ of $X(n,k,m)$ corresponding to the canonical class $K_{X(n,k)_{0}}$ satisfies $SW_{X(n,k,m)}(\beta_{n,k,m}) = SW_{X_{n,k}}(K_{X(n,k)}) + (m-1)SW_{X(n,k)_{0}} (K_{X(n,k)_{0}}) = 1 + (m-1) = m$. Thus, every $X(n,k,m)$ with $m \geq 2$ is nonsymplectic. 

\section{Construction of symplectic $4$-manifolds with various fundamental groups}

In this section we modify the above construction to obtain symplectic $4$-manifolds with the various finitely generated fundamental groups, such as the free groups of rank $s \geq 1$, and the finite free products of cyclic groups, using the \emph{Luttinger surgery}. Our construction can be generalized further to obtain symplectic $4$-manifolds with arbitrarily finitely presented fundamental groups and with \emph{small size}, but we will not pursue this here. We refer the reader to the articles in \cite{AZ} and \cite{AO} where this problem has been thoroughly studied using the \emph{Luttinger surgery}.    

Our first building block will be the symplectic $4$-manifold $Y(n,k) = \Sigma_{k} \times \mathbb{S}^2\#4n\CPb$ along with a regular fiber of the genus $2k+n-1$ Lefschetz fibration on $Y(n,k)$ (see Section~\ref{sbb}). We equip $Y(n,k) = \Sigma_{k} \times \mathbb{S}^2\#4n\CPb$ with the symplectic structure induced from the given Lefschetz fibration. The second building block will be the symplectic $4$-manifold $Y_{l}(1/p_{1},1/q_{1}, \cdots, 1/p_{l}, 1/q_{l})$ discussed in the Section \ref{L}, along with the symplectic submanifold $\Sigma_l$ (see Section~\ref{sbb}), where we set $l = 2k + n - 1$. To simplify the notation, we set $\overline{p}=(p_1, \ldots, p_l)$ and $\overline{q}= (q_1, \ldots, q_l)$ throughout this section. Let $X(n,k,\overline{p}, \overline{q})$ denote the symplectic $4$-manifold obtained by forming the symplectic fiber sum of $Y(n,k)$ and $Y_{l}(1/p_{1},1/q_{1}, \cdots, 1/p_{l}, 1/q_{l})$\/ along the surfaces $\Sigma_{2k+n-1}$ and $\Sigma_{l}$. Let $c_{1}$, $d_{1}$, $\cdots$, $c_{l}$, $d_{l}$, and $\mu' = [a_{1},b_{1}][a_{2},b_{2}]$ generate $\pi_1(\Sigma_{l}\times S^{1})$ in $Y_{l}(1/p_{1},1/q_{1}, \cdots, 1/p_{l}, 1/q_{l})\setminus\nu\Sigma_{l}$ (see Section \ref{sbb}). 

We choose the gluing diffeomorphism  $\psi : \Sigma_{2g+n-1}\times S^1 \rightarrow \Sigma_{l}\times S^1$ that maps the generators of the fundamental groups as follows: 
\begin{align*}
\psi_{\ast} (\alpha_{1}) = c_{1},\; 
\psi_{\ast} (\beta_{1}) = d_{1},\; 
\psi_{\ast} (\alpha_{2}) = c_{2},\; 
\psi_{\ast} (\beta_{2}) = d_{2},\; 
\cdots, 
\psi_{\ast} (\alpha_{2k}) = c_{2k},\;
\psi_{\ast} (\beta_{2k}) = d_{2k},\;
\\
\psi_{\ast} (e_{1}) = c_{2k+1},\;
\psi_{\ast} (e_{2}) = d_{2k+1},\;
\cdots, 
\psi_{\ast} (e_{2n-3}) = c_{2k+n-1},\;
\psi_{\ast} (e_{2n-2}) = d_{2k+n-1},\;
\psi_{\ast} (\mu) = \mu'
\end{align*} 
By Gompf's theorem in \cite{G}, $X(n,k,\overline{p},\overline{q})$ is symplectic. By Van Kampen's theorem, we have 

\begin{equation*}
\pi_{1}(X(n,k,\overline{p}, \overline{q})) \:=\: \frac{\pi_{1}(Y(n,k)\setminus \nu \Sigma_{2k+n-1})\ast \pi_{1}(Y_{l}(1/p_{1},1/q_{1}, \cdots, 1/p_{l}, 1/q_{l}) \setminus \nu \Sigma_{l})}{\langle \alpha_{1} = c_{1},\, \beta_{1} = d_{1},\, \cdots,\, \alpha_{2k} = c_{2k},\, \beta_{2k} = d_{2k},\, 
\cdots,\, e_{2n-2} = d_{2k+n-1},\, \mu = \mu' \rangle}.
\end{equation*}

Since the loops $e_{1}$, $e_{2}$, $\cdots$, $e_{2n-3}$, $e_{2n-2}$ and the normal circle to $\mu=\{{\rm pt}\}\times S^1$ are all nullhomotopic in $Y(n,k)\setminus\nu\Sigma_{2k+n-1}$, we get the following presentation for the fundamental group of $X(n,k, \overline{p}, \overline{q})$.

\begin{eqnarray}
\pi_1(X(n,k,\overline{p}, \overline{q}))  \nonumber &=& \langle 
c_{1}, d_{1}, \cdots, c_{2k}, d_{2k};\, a_{1}, b_{1}, a_{2}, b_{2} ;\, \mid \label{eq:pi_1(X(n,k))} \\ \nonumber
&&[b_1^{-1},d_1^{-1}]=a_1,\ \  [a_1^{-1},d_1]=b_1,\ \  [b_2^{-1},d_2^{-1}]=a_2,\ \  [a_2^{-1},d_2]=b_2,\\ \nonumber
&&[d_1^{-1},b_2^{-1}]=c_1^{p_1},\ \ [c_1^{-1},b_2]=d_1^{q_1},\ \ [d_2^{-1},b_1^{-1}]=c_2^{p_2},\ \ [c_2^{-1},b_1]=d_2^{q_2},\\ \nonumber
&&[a_1,c_1]=1, \ \ [a_1,c_2]=1,\ \  [a_1,d_2]=1,\ \ [b_1,c_1]=1, \\ \nonumber
&&[a_2,c_1]=1, \ \ [a_2,c_2]=1,\ \  [a_2,d_1]=1,\ \ [b_2,c_2]=1, \\ \nonumber
&&[a_1,b_1][a_2,b_2]=1,\ \ \prod_{j=1}^{2k}[c_j,d_j]=1, \\ \nonumber
&&[a_1^{-1},d_3^{-1}]=c_3^{p_3}, \ \ [a_2^{-1} ,c_3^{-1}] =d_3^{q_3}, \\ \nonumber
&& \cdots, \ \, \cdots, \\ \nonumber
&&[a_1^{-1},d_{2k}^{-1}]=c_{2k}^{p_l}, \ \ [a_2^{-1} ,c_{2k}^{-1}] =d_{2k}^{q_l},\\ \nonumber
&&[b_1,c_3]=1,\ \  [b_2,d_3]=1,\ \ \cdots, [b_1,c_{2k}]=1,\ \ [b_2,d_{2k}]=1, \\ \nonumber
&&c_{1}c_{2k} = 1,\ \  c_2c_{2k-1}=1,\ \  \cdots ,\ \  c_{k}c_{k+1}=1, \\ \nonumber
&&d_{1}d_{2}\cdots d_{2k} = 1,\ \  d_2d_{3}\cdots d_{2k-1}= [c_{2k},d_{2k}],\ \  \cdots , \\ \nonumber
&&d_{i+1}d_{i+2} \cdots d_{2k-i} = [c_{2k-i+1},d_{2k-i+1}] \cdots [c_{2k-1},d_{2k-1}][c_{2k},d_{2k}] \rangle. \\ \nonumber
\end{eqnarray}

To realize the free group of rank $s \geq 1$ as the fundamental groups, we simply set  $p_3 = \cdots = p_{2k} = 0$, $p_1 = p_2 = q_1 = \cdots = q_{2k} = 1$ in the above presentation. Using the identity ${c_{2k}}^{-1}=c_{1}$, we rewrite the relation $[a_{1}^{-1}, c_{2k}^{-1}]=d_{2k}$ as\/ $[a_{1}^{-1},c_{1}] = d_{2k}$. Since $[a_{1},c_{1}] = 1$, we obtain $d_{2k} = 1$. $d_{2k} = 1$ in turn implies $d_{1} = 1$ using the relations $d_2d_{3} \cdots d_{2k-1} = [c_{2k},d_{2k}]$ and $d_{1}d_{2}\cdots d_{2k} = 1$. Next, using $[b_1^{-1},d_1^{-1}]=a_1$, $[a_1^{-1},d_1]=b_1$ and $[d_1^{-1},b_2^{-1}]=c_1$, we obtain $a_{1} = b_{1} = c_{1} = 1$. Since $[c_2^{-1},b_1]=d_2$ and $[d_2^{-1},b_1^{-1}]=c_2$, we have $d_2 = c_2 = 1$, which in turn lead $a_2 = b_2 = 1$. Next, using the relations  $[a_{2}^{-1},c_i^{-1}]=d_{i}$ for any $3 \leq i \leq 2k$, we have $d_{3} = d_{4} = \cdots = d_{2k} = 1$. Since ${c_{2k-i+1}}^{-1}=c_{i}$ for any $i \leq k$ and $c_{1} = c_{2} = 1$, we conclude that $\pi_1(X(n,k,\overline{p},\overline{q}))$ is a free group of rank $s := k-2$ generated by $c_{3}$, $\cdots$, $c_{k}$.

To realize the finite free products of cyclic groups as the fundamental groups, we simply set $p_1 = p_2 = 1$, $p_3 = \cdots = p_l = 0$, and let $q_i \geq 1$ vary arbitrarily in the above presentation.

\section*{Acknowledgments} A. Akhmedov was partially supported by NSF grants FRG-0244663 and DMS-1005741. N. Saglam was partially supported by NSF grants FRG-0244663.

\end{document}